\documentclass[11pt]{article}
\usepackage{amsfonts,eucal,amsmath,bm}
\usepackage{a4wide,url}
\usepackage{amsfonts,amssymb,eucal,amsmath}
\newtheorem{theorem}{Theorem}

\newtheorem{lemma}{Lemma}

\def\R{\mathbb{R}}
\def\C{\mathbb{C}}

\def\Z{\mathbb{Z}}
\def\N{\mathbb{N}}

\def\cD{\mathcal{D}}
\def\cR{\mathcal{R}}

\def\cP{\mathcal{P}}
\newcommand{\ve}{\varepsilon}

\newenvironment{proof}{\noindent\textsc{{Proof}}.}{\hfill\raisebox{-1ex}{$\boxtimes$}}
\newcommand{\proofend}{$\boxtimes$}

\newcounter{remark}
\newcounter{example}
\newcounter{question}
\newcounter{problem}
\newcounter{conjecture}
\begin{document}
\newlength{\mylength}
\newcommand{\sui}{[-\tfrac12,\tfrac12]}

\large

\title{\bf Integral polynomials with  small \\ discriminants and resultants}

\author{Victor Beresnevich\footnote{Supported by EPSRC grant EP/J018260/1}\\ {\small\sc (York)} \and Vasili Bernik\\ {\small\sc (Minsk)} \and Friedrich G\"otze\footnote{Supported by SFB-701}\\ {\small\sc (Bielefeld)} }

\date{}

\maketitle

{\footnotesize

\begin{abstract}
Let $n\in\N$ be fixed, $Q>1$ be a real parameter and $\cP_n(Q)$ denote the set of polynomials over $\Z$ of degree $n$ and height at most $Q$. In this paper we investigate the following counting problems regarding polynomials with small discriminant $D(P)$ and pairs of polynomials with small resultant $R(P_1,P_2)$:
\begin{itemize}
  \item[(i)] {\em given $0\le v\le n-1$ and a sufficiently large $Q$, estimate the number of polynomials $P\in\cP_n(Q)$ such that }
  $$
  0<|D(P)|\le Q^{2n-2-2v}\,;
  $$
  \item[(ii)] {\em given $0\le w\le n$ and a sufficiently large $Q$, estimate the number of pairs of polynomials $P_1,P_2\in\cP_n(Q)$ such that }
  $$
  0<|R(P_1,P_2)|\le Q^{2n-2w}.
  $$
\end{itemize}
Our main results provide lower bounds within the context of the above problems.
We believe that these bounds are best possible as they correspond to the solutions of naturally arising linear optimisation problems. Using a counting result for the number of rational points near planar curves due to R.\,C.~Vaughan and S.~Velani we also obtain the complementary optimal upper bound regarding the discriminants of quadratic polynomials.
\end{abstract}



}

\section{Introduction}\label{itro}

Throughout this paper $n$ will denote a positive integer. In what follows, given a polynomial $P=a_nx^n+\dots+a_0\in\Z[x]$ of degree $n$, let
$$
H(P):=\max_{0\le i\le n}|a_i|
$$
denote the standard (naive) height of $P$ and, given a real parameter $Q>1$, let
\begin{equation}\label{e3}
\cP_n(Q)=\{P\in\Z[x]:\deg P=n,\ H(P)\le Q\}
\end{equation}
denote the set of integral polynomials $P$ of degree $n$ and height $H(P)\le Q$.
Throughout, $D(P)$ will stand for the discriminant of a polynomial $P$ and and $R(P_1,P_2)$ will stand for the resultant of polynomials $P_1,P_2$. The formal definitions and basic properties of these important number theoretic characteristics will be recalled below.

In this paper we investigate the following counting problems regarding the discriminant and resultant of polynomials in $\cP_n(Q)$.

\medskip

\noindent\textbf{Problem 1.} {\em Let $n\ge2$ be an integer. Given $0\le v\le n-1$ and a sufficiently large $Q$, estimate the number of polynomials $P\in\cP_n(Q)$ such that}
\begin{equation}\label{vv01}
  0<|D(P)|\le Q^{2n-2-2v}\,.
\end{equation}

\medskip

\noindent\textbf{Problem 2.} {\em Let $n\ge1$ be an integer. Given $0\le w\le n$ and a sufficiently large $Q$, estimate the number of pairs of polynomials $P_1,P_2\in\cP_n(Q)$ such that}
  $$
  0<|R(P_1,P_2)|\le Q^{2n-2w}\,.
  $$

\medskip

These natural problems of intrinsic interest originate from transcendental number theory. For instance, Davenport's estimate \cite{dav} for the sums of reciprocals to square roots of $|D(P)|$ was crucial to Volkmann's proof \cite{Volkmann2} of the cubic case of a long standing conjecture of Mahler on $S$-numbers \cite{b24}. Mahler's conjecture was eventually settled by Sprind\v zuk \cite{b27}. In it worth mentioning that properties of discriminants and resultant form the backbone of Sprind\v zuk's techniques as well as various generalisations, see for instance \cite{b1, b6, b7, b8, b12, b17, b22, b28}.

There are also $p$-adic and `mixed' analogues of the above problems which alongside the size of the discriminant and resultant address their arithmetic structure. More precisely, their formulation requires that the dicriminant/resultant is divisible by large powers of given prime numbers, see \cite{bbg} for an overview. In recent years there has been substantial activity in attempting to resolve Problems~1 and~2, see \cite{b4,b3,b9,b20,b21,b23}, as well as their $p$-adic versions, see \cite{b10,b13} and also \cite{bbg}.

The discriminant and resultant naturally encode the information regarding the distance between different algebraic numbers including conjugate algebraic numbers - see definitions \eqref{v103} and \eqref{v104} below. Thus Problems~1 and~2 naturally complement various questions regarding close algebraic numbers. Studying the latter dates back to a work of Mahler \cite{b25} who proved a general lower bound on the distance between two algebraic numbers. Establishing `correct' upper bounds as well as quantitative results have been the subject of numerous papers including \cite{b2,b14,b15,b16,b17,b18,b19,b26}, which thereby further motivate understanding the above problems.

\medskip

We now proceed by recalling some basic facts regarding discriminants and resultants. Given a polynomial
$
P=a_nx^n+\dots+a_0\in\Z[x]
$
of degree $n$, the discriminant of $P$ is defined by
\begin{equation}\label{v103}
D(P):=a_n^{2n-2}\prod_{1\le i<j\le n}(\alpha_i-\alpha_j)^2\,,
\end{equation}
where $\alpha_1,\dots,\alpha_n\in\C$ are the roots of $P$.
Given two polynomials
$$
P_1=a_nx^n+\dots+a_0\qquad\text{and}\qquad
P_2=b_mx^m+\dots+b_0
$$
of degrees $n$ and $m$ respectively, the resultant of $P_1$ and $P_2$ is defined to be
\begin{equation}\label{v104}
R(P_1,P_2):=a_n^mb_m^n\prod_{\substack{1\le i\le n\\ 1\le j\le m}}(\alpha_i-\beta_j)=
a_n^m\prod_{1\le i\le n}P_2(\alpha_i)=(-1)^{mn}R(P_2,P_1)\,,
\end{equation}
where $\alpha_1,\dots,\alpha_n\in\C$ are the roots of $P_1$ and $\beta_1,\dots,\beta_m\in\C$ are the roots of $P_2$.
It is clear that $D(P)=0$ if and only if $P$ has a repeated root, while $R(P_1,P_2)=0$ if and only if $P_1$ and $P_2$ have a common root.

It is well known and easily verified that
\begin{equation}\label{vb100}
D(P)=\frac{(-1)^{\frac{n(n-1)}{2}}R(P,P')}{a_n}\,,
\end{equation}
where $P'$ is the derivative of $P$.
Furthermore, there is a classical explicit formula for $D(P)$ and $R(P_1,P_2)$ via the determinant of a Sylvester matrix composed from the coefficients of the polynomials, see \cite{b29}. Namely, $R(P_1,P_2)$ can be computed as the following $(n+m)\times(n+m)$ determinant
\begin{equation}\label{res}
R(P_1,P_2)=\left|
\begin{array}{ccccccc}
  a_n & a_{n-1} & \dots & a_0 & 0 &\dots & 0\\
  0 & a_n & a_{n-1} & \dots & a_0  &\dots & 0\\
  \vdots &  & \ddots & \ddots &  & \ddots & \\
  0 & \dots & 0 & a_n & a_{n-1}  &\dots & a_0\\
  b_m & b_{m-1} & \dots & b_0 & 0 &\dots & 0\\
  0 & b_m & b_{m-1} & \dots & b_0  &\dots & 0\\
  \vdots &  & \ddots & \ddots &  & \ddots & \\
  0 & \dots & 0 & b_m & b_{m-1}  &\dots & b_0\\
\end{array}
\right|\,.
\end{equation}
The corresponding formula for $D(P)$ can be found from the above expression and \eqref{vb100}.
Two obvious consequences of \eqref{vb100} and \eqref{res} are that $D(P)$ is an integral polynomial of the coefficients of $P$, while $R(P_1,P_2)$ is an integral polynomial of the coefficients of $P_1$ and $P_2$. In particular, $D(P)$ and $R(P_1,P_2)$ return integer numbers for any choice of $P,P_1,P_2\in\Z[x]\setminus\{0\}$. Hence, whenever $D(P)\neq0$ we have that $|D(P)|\ge1$ and whenever $R(P_1,P_2)\neq0$ we have that $|R(P_1,P_2)|\ge1$ for all choices of non-zero integral polynomials $P,P_1,P_2$.

Another straightforward consequences of \eqref{res} is that $D(P)$ and $R(P_1,P_2)$ are bounded in terms of the heights and degrees of the polynomials. In particular, it is readily verified that for every $n\ge2$ there exists a constant $\gamma>0$ which depends on $n$ only such that for any $P\in\cP_n(Q)$ we have that
\begin{equation}\label{vv1}
|D(P)|\le \gamma Q^{2n-2}\,.
\end{equation}
This together with the inequality $|D(P)|\ge1$ clarifies why the range of $v$ is $[0,n-1]$ within the context of Problem~1.

Similarly, for any $n\in\N$ there exists a constant $\rho>0$ which depends on $n$ only such that for any two polynomials $P_1,P_2\in\cP_n(Q)$ we have that
$$
|R(P_1,P_2)|\le \rho Q^{2n}\,.
$$
In turn, this together with the inequality $|R(P_1,P_2)|\ge1$ clarifies why the range of $w$ is $[0,n]$ within the context of Problem~2.

Regarding Problem~1, in order to deal with the `extreme' case $v=n-1$ efficiently we will weaken inequality \eqref{vv01} by introducing the constant $\gamma$ the same way it appears in \eqref{vv1}. Thus, when addressing Problem~1 we will be rather estimating the size of the following subset of $\cP_n(Q)$:
\begin{equation}\label{e4}
\cD_{n,\gamma}(Q,v):=\{P\in\cP_n(Q):1\le |D(P)|\le\gamma Q^{2n-2-2v}\}\,.
\end{equation}
In the case $\gamma=1$, we will write $\cD_{n}(Q,v)$ for $\cD_{n,\gamma}(Q,v)$. Similarly, in order to deal with the extreme case $w=n$ within Problem~2 we will be estimating the size of
\begin{equation}\label{vb5}
\cR_{n,\rho}(Q,w)=\big\{(P_1,P_2)\in\cP_n(Q)^2:0<|R(P_1,P_2)|\le \rho Q^{2n-2w}\big\}\,,
\end{equation}
where $\rho$ a fixed constant, $0\le w\le n$ and $Q$ is sufficiently large.
In the case $\rho=1$, we will write $\cR_{n}(Q,w)$ for $\cR_{n,\rho}(Q,w)$.

In what follows we will use the Vinogradov symbol $\ll$\,. By definition, $X\ll Y$ means that $X\le CY$ for some constant $C$ which will depend on $n$ only. Also we will write $X\asymp Y$ if $X\ll Y$ and $Y\ll X$ simultaneously.

We now briefly recall the results that have been obtained to date. The first general estimate regarding Problem~1 was established in \cite{b9} by showing that
\begin{equation}\label{e5}
    \#\cD_{n}(Q,v)\gg Q^{n+1-2v}
\end{equation}
for $0<v<1/2$. In the case of quadratic polynomials it was shown in \cite{b21} that
\begin{equation}\label{vb2}
    \#\cD_{2,\gamma}(Q,v)=20(1+\ln2) Q^{3-2v}+O(Q^{3-3v}+Q^2)
\end{equation}
for if $0<v<1/2$.
This was obtained by calculating the volume of the body in $\R^3$ defined by the inequality $|a_1^2-4a_2a_0|< Q^{4-2v}$, which clearly contains all the integer vectors $(a_0,a_1,a_2)$ that define the polynomials of interest.

Once again, by calculating the volume of relevant bodies, this time in $\R^4$, it was shown in \cite{b20} that
\begin{equation}\label{vb1}
\#\cD_{3,\gamma}(Q,v)\asymp Q^{4-\frac53\,v}
\end{equation}
for $0<v<3/5$.
The latter estimate was in some way surprising as it led to the realisation of the fact that \eqref{vb2} is not in general sharp. One of the main goals of this paper is to make a further sept in determining the `right size' of $\cD_{n}(Q,v)$. The following main result of this paper extends the lower bound given by \eqref{vb1} to the full range of $v$ and to arbitrary degrees $n$.

\begin{theorem}\label{t1}
Let $n\in\N$, $n\ge2$ be given. Then there exists $\gamma>0$ depending on $n$ only such that for any sufficiently large $Q$ and $0\le v\le n-1$ one has that
\begin{equation}\label{vb3}
    \#\cD_{n,\gamma}(Q,v) \gg Q^{n+1-\frac{n+2}{n}v}\,,
\end{equation}
where the constant implied by the Vinogradov symbol depends on $n$ only.
\end{theorem}

\medskip

\noindent{\it Remark.} If we require that $v<n-1$, then $\gamma$ within the above result can be taken to be any constant, in particular $1$. This readily follows from the following trivial equality
$$
\cD_{n,\gamma}(Q,v)=\cD_{n,\gamma'}(Q,v')
$$
which holds whenever $\gamma' Q^{2v}=\gamma Q^{2v'}$. Since $\gamma$ and $\gamma'$ are fixed, when$v<n-1$ we will also have that $v'<n-1$ for sufficiently large $Q$. Hence \eqref{vb3} will imply the corresponding bound for $\#\cD_{n,\gamma'}(Q,v')$.

\medskip

Our second main result concerns resultants and improves upon the previously obtained lower bounds \cite{b3}.

\begin{theorem}\label{t2}
Let $n\in\N$ be given. Then there exists $\rho>0$ depending on $n$ only such that for any sufficiently large $Q$ we have that
\begin{equation}\label{vb3+}
    \#\cR_{n,\rho}(Q,w) \gg \left\{\begin{array}{ll}
                                 Q^{2n+2-2w}&\qquad\text{if }\ 0\le w\le \tfrac{n+1}2\,,\\[2ex]
                                 Q^{2n+2-2w-\frac{2}{n}(2w-n-1)}&\qquad\text{if }\ \tfrac{n+1}2\le w\le n\,,
                               \end{array}\right.
\end{equation}
where the constant implied by the Vinogradov symbol depends on $n$ only.
\end{theorem}

\medskip

\noindent{\it Remark.} If we require that $w<n$, then $\rho$ within the above result can be taken to be any constant, in particular $1$. The proof is similar to that given within the remark to Theorem~\ref{t1}.

\medskip

Theorems~\ref{t1} and \ref{t2} are obtained by constructing polynomials ({\em resp}. pairs of polynomials) in $\cP_n(Q)$ with a prescribed configuration of roots. This method is not new and, in view of \eqref{v103} and \eqref{v104}, is not surprising. Indeed, the distribution of roots instantly gives an estimate for the discriminant/resultant of polynomials.
However, previously the configuration of roots was designed by using `hands-on techniques' and unluckily resulted to not so sharp bounds for $\#\cD_n(Q,v)$ and $\#\cR_n(Q,w)$. In this paper we develop a general approach which enables us to determine the optimal configuration that maximises the number of choices for polynomials while keeping their discriminants/resultants under a given bound. Indeed, the approach boil down to a linear optimisation problem and the estimates for $\#\cD_n(Q,v)$ and $\#\cR_n(Q,w)$ we obtain correspond to optimal solutions to these problems. As a result, we believe that the estimates obtained in the above theorems are best possible, perhaps up to an arbitrarily small additive constant $\delta>0$ in the exponents within \eqref{vb3} and \eqref{vb3+}.

In should be noted that the case $v=0$ within Problem~1 corresponds essentially imposing no restriction on the discriminant. Naturally, our lower bound for $\#\cD_n(Q,v)$ in this case gives $\gg Q^{n+1}$, which is however quite trivial. Nevertheless, a much more detailed insight into the distribution of such (typical) values of the discriminant is provided in \cite{gz}. The main result of \cite{gz} proves as asymptotic formula for the number of polynomials $P\in\cP_n(Q)$ such that
$$
a\le \frac{|D(P)|}{Q^{2n-2}}\le b
$$
with a logarithmically small error term.

\section{Rational points near planar curves and the quadratic case}

Considering the case $n=2$ in this section we obtain the complementary upper bound for $\cD_2(Q,v)$ which holds for all $v\in(0,1)$ and thus extends \eqref{vb2} the full range of $v$.

\begin{theorem}\label{t4}
For any $v\in(0,1)$ and any sufficiently large $Q$ we have that
\begin{equation}\label{v101}
    \#\cD_{2}(Q,v) \asymp Q^{3-2v}\,.
\end{equation}
\end{theorem}

By Theorem~\ref{t1}, we only have to prove the upper bound, that is $\cD_{2}(Q,v)\ll Q^{3-2v}$. The proof will be based on the following counting result for rational points near planar curves due to Vaughan and Velani which is a straightforward consequence of Theorem~3 from \cite{VV06}. Note that the case $f(x)=x^2$ that will be of interest for us is also treated within \cite[Appendix 2 and Case (c) of \S2.1]{BDVV07}.

\begin{theorem}[Vaughan \& Velani \cite{VV06}]\label{t6}
Let $f:[\alpha,\beta]\to\R$ be a $C^{(3)}$ function. Suppose that $$\inf_{\alpha\le x\le \beta}|f''(x)|>0.$$
Let
$$
N(T,\ve):=\#\{(a,q)\in\Z\times\N: q\le T,\ \alpha q<a\le \beta q, \|qf(a/q)\|<\ve\}\,,
$$
where $\|\cdot\|$ denote the distance to the nearest integer. Then for any $T\ge1$, $\ve\in(0,\tfrac12)$ and any $\delta>0$ one has that
$$
N(T,\ve)~\ll~\ve T^2+\ve^{-1/2}T^{1/2+\delta}+\ve^{-\delta}T^{1+\delta}\,,
$$
where the implied constant in the Vinogradov symbol depends on $f$ and $\delta$ only.
\end{theorem}

\bigskip

\noindent\textit{Proof of Theorem~\ref{t4}.} First of all note that there is no loss of generality in assuming that $v>0$ as otherwise the statement of Theorem~\ref{t4} is trivial.
Note that any polynomial $P\in\cP_2(Q)$ is of the form
$$
P=ax^2+bx+c\,,\qquad a,b,c\in\Z, \ a\neq0,\ \max\{|a|,|b|,|c|\}\le Q\,.
$$
Thus the cardinality of $\cD_{2}(Q,v)$ is bounded by that of the set
\begin{equation}\label{vv2}
\{(a,b,c)\in\Z^3:~\max\{|a|,|b|,|c|\}\le Q,~0<|b^2-4ac|\le Q^{2-2v}\}\,.
\end{equation}
We shall consider two cases depending on the size of the coefficients $a,b,c$.

\bigskip

\noindent\textsf{Case 1\,:} $\max\{|a|,|c|\}\le \tfrac13|b|$. Then, $|4ac|<\tfrac12|b|^2$ and therefore $|b^2-4ac|\ge\tfrac12|b|^2$. Since we are interested in the triples $(a,b,c)$ lying in \eqref{vv2}, we get that $|b|\ll Q^{1-v}$. Since $\max\{|a|,|b|\}\le \tfrac13|b|$, we also have that $\max\{|a|,|b|\}\ll Q^{1-v}$. Thus, the number of triples in question is $\ll \left(Q^{1-v}\right)^3=Q^{3-3v}\le Q^{3-2v}$ and we obtain the required bound.

\bigskip

\noindent\textsf{Case 2\,:} $\max\{|a|,|c|\}> \tfrac13|b|$. Since the discriminant $b^2-4ac$ does not change if we swap $a$ and $c$, without loss of generality we can assume that $|a|=\max\{|a|,|c|\}$. In particular, we have that $\tfrac13 |b|\le |a|\le Q$.
Next, for each $a$ as above there exists an integer $t\ge0$ such that $2^t\le Q$ and $2^t\le|a|< 2^{t+1}$. Then dividing the inequality in \eqref{vv2} through by $a$ we get that
\begin{equation}\label{vv3}
|a(b/a)^2-4c|\ll \ve_t:=2^{-t}Q^{2-2v}\,.
\end{equation}

If $\ve_t<\tfrac12$, then using Theorem~\ref{t6} with
\begin{align*}
&f(x)=x^2,&& T=2^t &\ve=\ve_t\\
&[\alpha,\beta]=[0,3],
&&\delta=\min\{3-3v,v^{-1}-1\}\,,
\end{align*}
where the choice of $\delta$ is justified by the conditions $0<v<1$,
we conclude that the number of triples $(a,b,c)$ in \eqref{vv2} such that $|a|=\max\{|a|,|c|\}> \tfrac13|b|$ and $2^t\le |a|<2^{t+1}$ is
\begin{equation}\label{vv4}
\ll \ve_t 2^{2t}+\ve_t^{-1/2}2^{t(1/2+\delta)}+\ve_t^{-\delta}2^{t(1+\delta)}
~\ll~ 2^{t}Q^{2-2v}+2^{t}Q^{-1+v}2^{t\delta}+Q^{2v\delta}2^{t}\,.
\end{equation}
This estimate also holds when $\ve_t\ge\frac12$. Indeed, for $\ve_t\ge\tfrac12$ the number of different integers $c$ satisfying \eqref{vv3} is $\ll \ve_t$. Also the number of different integer pairs $(a,b)$ such that $2^t\le |a|<2^{t+1}$ and $\tfrac13|b|\le |a|$ is $\ll 2^{2t}$. Hence, when $\ve_t\ge\tfrac12$ the number of triples $(a,b,c)$ in question is $\ll 2^tQ^{2-2v}$.

Summing \eqref{vv4} over non-negative integers $t$ such that $2^t\le Q$ gives the following estimate for the number of triples $(a,b,c)$ in question:
$$
\ll Q^{3-2v}+Q^{v+\delta}+Q^{1+2v\delta}
$$
which is $\ll Q^{3-2v}$ since $\delta=\min\{3-3v,v^{-1}-1\}$. The proof is thus complete.

\vspace*{-1ex}
\hfill\proofend

\section{Auxiliary lemmas}

We begin with the following statement which is a version of Lemma~4 from \cite{b2}. In what follows, given a Lebesgue  measurable set $X\subset\R^d$, $|X|$ will stand for its (ambient) Lebesgue measure.

\begin{lemma}\label{l:04}
Let $n\ge 2$ and $v_0,\dots,v_n$ be a collection of real numbers such that
\begin{equation}\label{vvv1}
v_0+\ldots+v_n=0
\end{equation}
and that
\begin{equation}\label{vvv2}
v_0\ge v_1\ge \ldots \ge v_n=-1.
\end{equation}
Then there are positive constants $\delta_0$ and $c_0$ depending on $n$ only with the following property. For any interval $J\subset\sui $ there is a sufficiently large $Q_0$ such that for all $Q>Q_0$ there is a measurable set $G_J\subset J$ satisfying
\begin{equation}\label{e:024}
    |G_J|\ge \tfrac34|J|
\end{equation}
such that for every $x\in G_J$ there are $n+1$ linearly independent primitive irreducible polynomials $P\in\Z[x]$ of degree exactly $n$ such that
\begin{equation}\label{e:041}
\left\{\begin{array}{rcl}
\delta_0 Q^{-v_0} ~\le &|P(x)|  &\le~ c_0 Q^{-v_0}\,,\\[1ex]
\delta_0 Q^{-v_j} ~\le &|P^{(j)}(x)|  &\le~  c_0 Q^{-v_j}\qquad (1\le j\le n).
\end{array}\right.
\end{equation}
\end{lemma}

\bigskip

\noindent\textit{Remark.} Note the constant $\tfrac34$ is not crucial for the above lemma and can be replaced with any other positive constant strictly less than $1$.

\medskip

We now establish a general statement that relates the derivatives of $P$ and its roots.

\begin{lemma}\label{l3b}
Let $x\in\C$ be a fixed point, $P$ be a polynomial with complex coefficients of degree $\deg P=n>0$ and let $\alpha_1,\dots,\alpha_n$ be the roots of $P$ ordered so that
\begin{equation}\label{e13}
    |x-\alpha_1|\le|x-\alpha_2|\le\ldots\le|x-\alpha_n|\,.
\end{equation}
Let $a_n$ be the leading coefficient of $P$. Then for $0\le j< n$ we have that
\begin{equation}\label{vb6}
\textstyle |P^{(j)}(x)|\le \binom{n}{j}\,|a_n| \, |x-\alpha_{j+1}|\cdots|x-\alpha_n|,
\end{equation}
where as usually $\binom{n}{j}=\frac{n!}{j!(n-j)!}$ is a binomial coefficient.
If we further assume that for some $j>0$
\begin{equation}\label{vb8}
\textstyle|x-\alpha_{j}|<\frac12\binom{n}{j}^{-1}|x-\alpha_{j+1}|\,,
\end{equation}
then we also have that
\begin{equation}\label{vb7}
\textstyle |P^{(j)}(x)|\ge \frac12|a_n| \, |x-\alpha_{j+1}|\cdots|x-\alpha_n|\,.
\end{equation}

\end{lemma}

\begin{proof}
First observe that
\begin{equation}\label{vb6a}
P^{(j)}(x)=\hspace{1ex}a_n\hspace*{-4ex}\sum_{1\le i_1<\dots<i_{n-j}\le n}(x-\alpha_{i_1})\cdots(x-\alpha_{i_{n-j}})\,.
\end{equation}
Note that the right hand side contains exactly $\binom{n}{j}$ terms. By \eqref{e13}, the largest of these terms in absolute value is $T=(x-\alpha_{j+1})\cdots(x-\alpha_n)$ and so \eqref{vb6} readily follows from \eqref{vb6a}.

Further, \eqref{vb7} is trivially true when $j=0$ (it actually becomes an equality when the factor $\frac12$ is removed). When \eqref{vb8} is satisfied for some $j>0$, then any term of the sum in the right hand side of \eqref{vb6a} different from $T$ is $\le \tfrac12\binom{n}{j}^{-1}|T|$. Hence,
$$
\left|\sum_{1\le i_1<\dots<i_{n-j}\le n}(x-\alpha_{i_1})\dots(x-\alpha_{i_{n-j}})-T\right|\le \textstyle\frac12|T|
$$
whence \eqref{vb7} readily follows upon using the triangle inequality.
\end{proof}

\bigskip
\bigskip

\noindent Now we apply Lemma~\ref{l3b} to inequalities \eqref{e:041} to obtain the following

\begin{lemma}\label{l2}
Let $n$, $m$ and $v_j$ be the same as in Lemma~\ref{l:04}. Let
\begin{equation}\label{e12}
    d_j=v_{j-1}-v_j \quad(1\le j\le n).
\end{equation}
Suppose that
\begin{equation}\label{ddd}
 d_1\ge d_2\ge \ldots \ge d_n\ge0
\end{equation}
and that for some $x\in\C$ and $Q>1$ inequalities \eqref{e:041} are satisfied by some polynomial $P$ over $\C$ of degree $\deg P=n$. Then there are roots $\alpha_1,\dots,\alpha_{n}\in\C$ of $P$ such that
\begin{equation}\label{e14}
 |x-\alpha_j|  \le c_j Q^{-d_j}\qquad (1\le j\le n)\,,
\end{equation}
where
$$
c_1=nc_0\delta_0^{-1}\quad\text{and}\quad c_{j+1}=\max\left\{\frac{2c_0}{\delta_0}\binom{n}{j+1}, 2c_j\binom{n}{j}\right\} \quad(1\le j\le n-1).
$$
\end{lemma}

\begin{proof}
Since $P$ is of degree $n$ it has $n$ complex roots, say $\alpha_1,\dots,\alpha_n$. Order the roots with respect to their distance to $x$, thus ensuring the validity of \eqref{e13}.

First take $j=1$. Then, by Lemma~\ref{l3b}, we have that
$$
|P'(x)|\le n|a_n|  |x-\alpha_{2}|\cdots|x-\alpha_n|=n\frac{|P(x)|}{|x-\alpha_1|}\,.
$$
Then, by \eqref{e:041}, we obtain that
$$
|x-\alpha_1|\le nc_0\delta_0^{-1}Q^{-v_0+v_1}= nc_0\delta_0^{-1}Q^{-d_1}
$$
as required. We proceed by induction. Suppose \eqref{e14} holds for some $j< n$ and we now want to prove it for $j+1$.
If \eqref{vb8} is fulfilled, then we are in a position similar to the case $j=1$. Namely, by Lemma~\ref{l3b}, we have \eqref{vb7} while, by \eqref{vb6} written for $P^{(j+1)}(x)$, we have
\begin{equation}\label{vb6v}
\textstyle |P^{(j+1)}(x)|\cdot|x-\alpha_{j+1}|\le \binom{n}{j+1}|a_n|  |x-\alpha_{j+1}|\cdots|x-\alpha_n|.
\end{equation}
This together with \eqref{vb7} now gives
$$
\textstyle |P^{(j+1)}(x)|\cdot|x-\alpha_{j+1}|\le 2\binom{n}{j+1}|P^{(j)}(x)|.
$$
By \eqref{e:041}, we get
$$
\textstyle |x-\alpha_{j+1}|\le 2c_0\binom{n}{j+1}\delta_0^{-1}Q^{-v_j+v_{j+1}}=2c_0\binom{n}{j+1}\delta_0^{-1}Q^{-d_{j+1}}\le c_{j+1}Q^{-d_{j+1}}.
$$
On the other hand, if \eqref{vb8} is not fulfilled, then, $|x-\alpha_{j+1}|\le 2\binom{n}{j}|x-\alpha_{j}|$. Therefore, using induction and \eqref{ddd} we obtain that
$$
\textstyle|x-\alpha_{j+1}|\le 2\binom{n}{j}c_jQ^{-d_j}\le
2\binom{n}{j}c_jQ^{-d_{j+1}}\le c_{j+1}Q^{-d_{j+1}}.
$$
This completes the proof.
\end{proof}

\bigskip

The last auxiliary statement, which will only be used in the proof of Theorem~\ref{t2}, concerns measurable sets in the plane lying near the line $y=x$.

\begin{lemma}\label{l4}
Let $I\subset \R$ be a finite interval, $0<\Delta<|I|$ and
$$
X_\Delta=\{(x,y)\in\R^2:|x-y|\le\Delta\}.
$$
Let $A$ be a Lebesgue measurable subset of $I$ such that $|A|\ge \lambda|I|$ for some $\lambda\in(0,1)$. Let $A^2:=\{(x,y)\in\R^2:x,y\in A\}$. Then
$$
|A^2\cap X_\Delta|\ge  2^{-6}\lambda^3\Delta\,|I|.
$$
\end{lemma}

\begin{proof}
Let $T$ be an integer such that
\begin{equation}\label{vbvb}
\frac{|I|}{\Delta}<T< \frac{2|I|}{\Delta}\,.
\end{equation}
The existence of $T$ follows from the inequality $\Delta<|I|$.

Divide $I$ into $T$ equal subintervals $I_1,\dots,I_T$. By \eqref{vbvb}, $\tfrac12\Delta\le |I_j|\le\Delta$. Let $N$ be the number of intervals $I_j$ such that $|A\cap I_j|\ge \tfrac14\lambda\Delta$. Then
$$
|A|\le \tfrac14\lambda\Delta T+N\Delta.
$$
On the other hand, using \eqref{vbvb} we get that
$$
|A|\ge\lambda|I|\ge \tfrac12\lambda\Delta T.
$$
Hence $\tfrac12\lambda\Delta T\le \tfrac14\lambda\Delta T+N\Delta$ whence we obtain that $N\ge\tfrac14\lambda T$.

Since $|I_j|\le\Delta$ we trivially have that $I_j^2\subset X_\Delta$. Hence, for each $j$ we have the inclusion
$$
(A\cap I_j)^2\subset A^2\cap X_\Delta.
$$
In particular, when $|(A\cap I_j)^2|\ge \tfrac1{16}\lambda^2\Delta^2$ whenever $|A\cap I_j|\ge \tfrac14\lambda\Delta$. Recall that we have $N\ge\tfrac14\lambda T$ intervals $I_j$ satisfying this condition. Hence,
$$
|A^2\cap X_\Delta|\ge  N\cdot \tfrac1{16}\lambda^2\Delta^2\ge 2^{-6}\lambda^3 T\Delta^2~\stackrel{\eqref{vbvb}}{\ge}~ 2^{-6}\lambda^3 \frac{|I|}{\Delta}\,\Delta^2= 2^{-6}\lambda^3 |I|\Delta
$$
as required.
\end{proof}

\section{Proof of Theorem~\ref{t1}}

Let $v_0,\dots,v_n$ be given and satisfy \eqref{vvv1}, \eqref{vvv2} and let the parameters $d_j$ be given by \eqref{e12} and satisfy \eqref{ddd}. First of all let us show that
\begin{equation}\label{jdj}
    \sum_{j=1}^njd_j=n+1\,.
\end{equation}
Indeed, by \eqref{e12}, we have that $v_{j-1}=d_j+v_j$. Hence $v_{j-1}=d_j+\dots+d_n+v_n$. Since $v_n=-1$ we have that $v_{j-1}+1=d_j+\dots+d_n$. Also $v_n+1=0$. Summing these equations over $j=0,\dots,n$, by \eqref{vvv1}, we get
$$
n+1=\sum_{j=0}^n(v_j+1)=\sum_{j=0}^{n-1}(v_j+1)=\sum_{j=0}^{n-1}(d_{j+1}+\dots+d_n)=\sum_{j=1}^njd_j
$$
as stated in \eqref{jdj}.

Now, let $J=[-\tfrac12,\tfrac12]$, $Q$ be sufficiently large and $x\in G_J$, where $G_J$ is the same as in Lemma~\ref{l:04}. By Lemma~\ref{l:04}, inequalities \eqref{e:041} are satisfied for some irreducible polynomial $P\in\Z[x]$ of degree $n$. Then, by Lemma~\ref{l2}, we have \eqref{e14}. Hence, for any pair of integers $(i,j)$ satisfying $1\le i<j\le n$ we have that
$$
|\alpha_i-\alpha_j|\le |x-\alpha_i|+|x-\alpha_j|\ll Q^{-d_j}\,.
$$
By \eqref{e:041}, we have that
 \begin{equation}\label{height0}
 H(P)\le h_0Q
\end{equation}
for some constant $h_0$ which depends on $n$ only. Therefore, using \eqref{v103} we conclude that
\begin{equation}\label{vbcc}
1\le |D(P)|\ll Q^{2n-2}\prod_{1\le i<j\le n} Q^{-2d_j}=Q^{2n-2-2\sum_{j=2}^n(j-1)d_j}\,.
\end{equation}
Note that the left hand side inequality is due to the irreducibility of $P$.
The goal is to construct polynomials with
\begin{equation}\label{vcvc}
1\le |D(P)|\le \gamma Q^{2n-2-2v}
\end{equation}
with some suitably chosen constant $\gamma$. By \eqref{vbcc}, inequalities \eqref{vcvc} are fulfilled
if we impose the condition $\sum_{j=2}^n(j-1)d_j=v$. Subtracting this from \eqref{jdj} we obtain the equivalent equation
\begin{equation}\label{eq1}
    \sum_{j=1}^nd_j=n+1-v\,.
\end{equation}

Now, we estimate the number of polynomials that we can obtain this way. By \eqref{e14} and Lemma~\ref{l:04}, for every $x\in G_J$ we have that $|x-\alpha_1(P)|\ll Q^{-d_1}$, where $P$ arises from Lemma~\ref{l:04}. Therefore, since \eqref{vcvc} holds whenever \eqref{eq1} is satisfied, we have that
$$
G_J\subset \bigcup_{P\in\cD_{n,\gamma}(h_0Q,v)}\ \ \bigcup_{j=1}^n\big\{|x-\alpha_j(P)|\le c_1 Q^{-d_1}\big\}\,,
$$
where $\gamma$ and $c_1$ depend on $n$ only.
Here we have used bound \eqref{height0} on the height. Hence
$$
\tfrac34=\tfrac34|J|\le 2nc_1 Q^{-d_1}\#\cD_{n,\gamma}(h_0Q,v)
$$
and we get that
$$
\#\cD_{n,\gamma}(h_0Q,v)\gg Q^{d_1}\,.
$$

To obtain the best lower bound we should maximise $d_1$ subject to conditions \eqref{ddd}, \eqref{jdj} and \eqref{eq1}. This is a linear optimization problem. By \eqref{jdj}, $d_1$ would be maximal if we could ensure that $d_2=\dots=d_n$.
Then \eqref{jdj} and \eqref{eq1} give a system of two linear equations with two variables, namely $d_1$ and $d_2$, from which we easily find that
$$
d_2=\dots=d_n=\frac{2v}{n(n-1)}\qquad\text{and}\qquad d_1=n+1-\frac{n+2}{n}v.
$$
A quick check shows that $d_1\ge d_2$ because $v\le n-1$. Thus \eqref{ddd} is also fulfilled.
One can also verify that
$$
v_0=n-v\qquad\text{and}\qquad v_j=-1+(n-j)d_2\quad (1\le j\le n-1).
$$
Thus,
$$
\#\cD_{n,\gamma}(h_0Q,v)\gg Q^{d_1}\ge Q^{n+1-\frac{n+2}{n}v}\,.
$$
To complete the proof it remains to rescale the bound on the height by making the following change of variables $\widetilde Q=h_0Q$ in the above expression. This results in the required lower bound
$$
\#\cD_{n,\gamma}(\widetilde Q,v)\gg {\widetilde Q}\,^{n+1-\frac{n+2}{n}v}
$$
and completes the proof.

\section{Proof of Theorem~\ref{t2}}

The start is the same as in the proof of Theorem~\ref{t1}. Let $v_0,\dots,v_n$ be given and satisfy \eqref{vvv1} and \eqref{vvv2} and let the parameters $d_j$ be given by \eqref{e12} and satisfy \eqref{ddd}. Equation \eqref{jdj} is then again satisfied.

Let $J=[-\tfrac12,\tfrac12]$, $Q$ be sufficiently large, $G_J$ be the same as in Lemma~\ref{l:04} and $(x_1,x_2)\in G_J^2\cap X_{Q^{-t}}$, where $X_{Q^{-t}}$ is $X_\Delta$ with $\Delta=Q^{-t}$ as defined in Lemma~\ref{l4}. By Lemma~\ref{l:04}, for each $i=1,2$ there is an irreducible polynomial $P_i$ over $\Z$ of degree $n$ such that inequalities \eqref{e:041} with $x=x_i$ are satisfied. Furthermore, by Lemma~\ref{l4}, we can assume that $P_1$ and $P_2$ are linearly independent and therefore coprime. Since they are also irreducible, $P_1$ and $P_2$ have no common roots and therefore $|R(P_1,P_2)|\ge1$. We will impose the condition
\begin{equation}\label{tt}
    d_1\ge t\ge d_2.
\end{equation}

By Lemma~\ref{l2}, there are roots $\alpha_1,\dots,\alpha_{n}\in\C$ of $P_1$ such that
\begin{equation}\label{e14a}
 |x_1-\alpha_i|  \le c_i Q^{-d_i}\qquad (1\le i\le n)
\end{equation}
and  there are roots $\beta_1,\dots,\beta_n\in\C$ of $P_2$ such that
\begin{equation}\label{e14b}
 |x_2-\beta_j|  \le c_j Q^{-d_j}\qquad (1\le j\le n).
\end{equation}
Hence, for any pair of integers $(i,j)$ such that $1\le i,j\le n$ using the triangle inequality we obtain that
$$
|\alpha_i-\beta_j|\le |x_1-x_2|+|x_1-\alpha_i|+|x_2-\beta_j|\ll Q^{-\min\{t,d_i,d_j\}}\,.
$$
By \eqref{e:041}, we have that
 \begin{equation}\label{height}
 H(P_i)\le h_0Q\qquad (i=1,2)
\end{equation}
for some constant $h_0$ which depends on $n$ only. Therefore, by \eqref{v104}, we get that
$$
1\le |R(P_1,P_2)|\ll Q^{2n}\prod_{1\le i,j\le n} Q^{-\min\{t,d_i,d_j\}}\,.
$$
In view of \eqref{ddd} and \eqref{tt} we then have that
\begin{equation}\label{vfvf1}
1\le |R(P_1,P_2)|\ll Q^{2n-t-d_2-\dots-d_n-2\sum_{1\le i<j\le n}d_j}\,.
\end{equation}
The goal is to construct pairs of polynomials $P_1,P_2$ with
\begin{equation}\label{vfvf2}
1\le |R(P_1,P_2)|\le\rho Q^{2n-2w}
\end{equation}
with some suitably chosen constant $\rho$.
By \eqref{vfvf1}, inequalities \eqref{vfvf2} are fulfilled if we impose the condition
\begin{equation}\label{vb4}
t+d_2+\dots+d_n+2\sum_{1\le i<j\le n}d_j=2w.
\end{equation}
Using \eqref{jdj} we get that
$$
\sum_{1\le i<j\le n}d_j=\sum_{j=1}^n(j-1)d_j=n+1-\sum_{j=1}^nd_j.
$$
Hence \eqref{vb4} transforms into
\begin{equation}\label{vb4+}
2d_1-t+\sum_{j=2}^nd_j=2n+2-2w.
\end{equation}
Now, we estimate the number of pairs $(P_1,P_2)$ that we can obtain this way. By Lemma~\ref{l4}, we have that $|G_J^2\cap X_{Q^{-t}}|\gg Q^{-t}|J|=Q^{-t}$. Next, by Lemma~\ref{l:04}, \eqref{height} and \eqref{vfvf2}, we have that
$$
G_J^2\cap X_{Q^{-t}}\subset \bigcup_{(P_1,P_2)\in\cR_{n,\rho}(h_0Q,w)}\ \ \bigcup_{i,j=1}^n\left\{(x_1,x_2):\begin{array}{l}
                                  |x_1-\alpha_i(P_1)|\le c_1 Q^{-d_1}\\[0.5ex]
                                  |x_2-\beta_j(P_2)|\le c_1Q^{-d_1}
\end{array}
\right\}\,,
$$
where $\rho$ and $c_1$ depend on $n$ only.
Hence
$$
Q^{-t}\ll |G_J^2\cap X_{Q^{-t}}|\ll Q^{-2d_1}\#\cR_{n,\rho}(h_0Q,w)
$$
and we get that
$$
\#\cR_{n,\rho}(h_0Q,w)\gg Q^{2d_1-t}\,.
$$

To obtain the best lower bound we should maximaze $2d_1-t$ subject to conditions \eqref{ddd}, \eqref{jdj}, \eqref{tt} and \eqref{vb4+}. This is again a linear optimization problem. By \eqref{jdj}, $2d_1-t$ would be maximal if we could ensure that $d_2=\dots=d_n$. Assuming these equations, \eqref{ddd}, \eqref{jdj}, \eqref{tt} and \eqref{vb4+} give
\begin{equation}\label{sys}
\begin{array}{l}
2d_1+(n-1)(n+2)d_2=2n+2,\\[1ex]
(n^2-1)d_2+t=2w,\\[1ex]
d_1\ge t\ge d_2\,.
\end{array}
\end{equation}
Thus, we we have to maximize $2d_1-t$, bearing in mind that $0<w\le n$ together with the above constrains.
In particular, the first two equations of \eqref{sys} imply that
\begin{equation}\label{sys2}
  2d_1-t+(n-1)d_2=2n+2-2w\,.
\end{equation}
Clearly, the best we can get is $2d_1-t=2n+2-2w$ in the case $d_2=0$. However, this does not a priori mean that the last condition of \eqref{sys} is fulfilled. For this reason we are forced to consider the following two cases.

\bigskip

\noindent\textsf{Case (i):} $w\le \tfrac{n+1}2$. Then we can indeed take $d_2=0$, $d_1=n+1$ and $t=2w$.
Clearly, the last condition of \eqref{sys} holds. In this case
$$
v_0=n, \quad v_j=-1\qquad (1\le j\le n-1)
$$
and we get that
$$
\#\cR_{n,\rho}(h_0Q,w)\gg Q^{2d_1-t}\ge Q^{2n+2-2w}\,.
$$

\bigskip

\noindent\textsf{Case (ii):} $\tfrac{n+1}2<w\le n$. In this case $d_2=0$ would imply via \eqref{sys} that $t>n+1>d_1$, contrary to the requirement $d_1\ge t$. It is easily calculated from \eqref{sys} that the smallest value of $d_2$ which enables the condition $t\le d_1$ is
$$
d_2=\frac{4w-2n-2}{n(n-1)}\,.
$$
In view of \eqref{sys2} this maximises $2d_1-t$. Then, from \eqref{sys} we obtain that
$$
t=d_1\qquad\text{and}\qquad d_1=2n+2-2w-\tfrac{2}{n}(2w-n-1)\,.
$$
A quick check shows that $0\le d_2\le d_1$ since $\tfrac{n+1}2<w\le n$.
Thus, all the required conditions are met and we have that
$$
\#\cR_{n,\rho}(h_0Q,w)\gg Q^{2d_1-t}=Q^{d_1}\ge Q^{2n+2-2w-\frac2n(2w-n-1)}\,.
$$
Finally, to complete the proof it remain to re-scale the bound on the height the same way as we did in the proof of Theorem~\ref{t1}, that is by setting $\widetilde Q=h_0Q$.

\section{Comparing the estimates for different degrees}

Given $n\in\N$, $n\ge2$, $Q>1$ and $x\ge0$, consider the set
$$
\cD_{\le n}(Q,x):=\big\{P\in\Z[x]:2\le \deg P\le n, ~H(P)\le Q,~1\le |D(P)|\le \gamma Q^{x}\big\}\,.
$$
This set is composed of the polynomials of degree up to $n$ with a given restriction on the discriminant.
By Theorem~\ref{t1}, there is a $\gamma>0$ which depends on $n$ only such that the number of polynomials $P\in\cP_k(Q)$ lying in this set is
$$
\gg Q^{k+1-\frac{k+2}{k}v}\,,
$$
where $v$ is determined from the equation $x=2k-2-2v$ when $x\le 2k-2$ and $v=0$ when $x>2k-2$.
Let
$$
f_k(x)=k+1-\frac{k+2}{k}\Big(k-1-\min\{k-1,x/2\}\Big)
$$
$$
=\left\{\begin{array}{ccl}
                                                     \tfrac2k+\frac{k+2}{2k}x & \text{if}& 0<x\le 2k-2\,, \\[1ex]
                                                     k+1 & \text{if}& ~~~~~\,x> 2k-2\,.
                                                   \end{array}
\right.
$$
Then the number of polynomials $P\in\cP_k(Q)$ lying in $\cD_{\le n}(Q,x)$ is
$
\gg Q^{f_k(x)}\,.
$
Consequently,
$$
\cD_{\le n}(Q,x)\gg Q^{d_n(x)}\,,
\qquad\text{where}\qquad
d_n(x):=\max\limits_{2\le k\le n}f_k(x)\,.
$$
Since the slopes of the lines $y_k=\tfrac2k+\frac{k+2}{2k}x$ and the points of their intersection with the $y$-axis get smaller as $k$ gets bigger, the graph of $y_k$ lies under the graph of any other line $y_m$ with $m<k$. Hence, $f_k$ will always intersect $d_{k-1}$ at some $x>2k-4$. Hence the contribution to $\cD_{\le n}(Q,x)$ by polynomials of degree $k$ will outweigh only when the contribution by polynomials of smaller degrees is no longer growing. This may seem rather counterintuitive as there are generally many more polynomials of higher degree. Below we sketch the graph of $d_n(x)$ for the case $n=4$, which is enough to exhibit the `staircase' nature of this function.

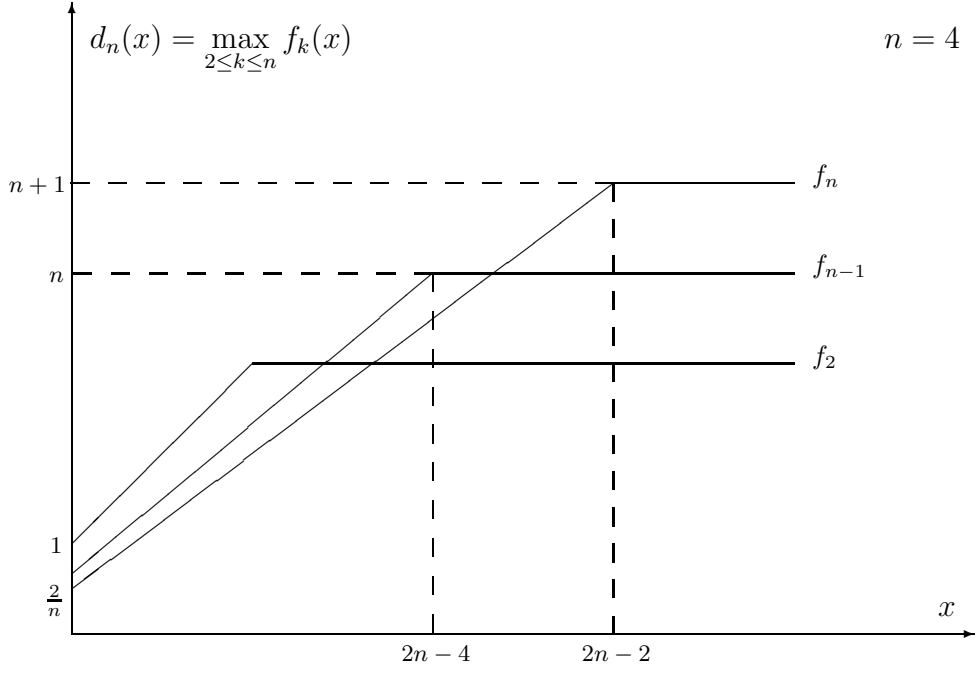
\setlength{\unitlength}{1.2cm}
\begin{picture}(11,8)
\put(0,0){\vector(0, 1){7}}
\put(0,0){\vector(1, 0){10}}

%k=2
\put(2,3){\line(-1, -1){2} }
\put(2,3){\line(1, 0){6} }
\put(8.2,3){{\small$f_2$}}

%k=3
\put(4,4){\line(-6, -5){4} }
\put(4,4){\line(1, 0){4} }
\multiput(4,4)(0,-0.42){10}{\line(0,-1){0.2}}
\put(3.65,-0.3){{\footnotesize$2n-4$}}
\multiput(4,4)(-0.42,0){10}{\line(-1,0){0.2}}
\put(-0.25,3.9){{\footnotesize$n$}}
\put(8.2,4){{\small$f_{n-1}$}}

%k=4
\put(6,5){\line(-4, -3){6} }
\put(6,5){\line(1, 0){2} }
\multiput(6,5)(0,-0.40){13}{\line(0,-1){0.2}}
\put(5.65,-0.3){{\footnotesize$2n-2$}}
\multiput(6,5)(-0.415,0){15}{\line(-1,0){0.2}}
\put(-0.7,4.9){{\footnotesize$n+1$}}
\put(8.2,5){{\small$f_n$}}

\put(9.6,0.2){$x$}
\put(0.2,6.5){$d_n(x)=\max\limits_{2\le k\le n}f_k(x)$}
\put(9.0,6.5){$n=4$}

\put(-0.24,0.9){{\footnotesize$1$}}
\put(-0.3,0.3){{\small$\tfrac2n$}}

\end{picture}

\bigskip
\bigskip

The following explicit formula for $d_n(x)$ is also readily computed:
$$
d_n(x)=\left\{\begin{array}{cll}
              x+1 & \text{ if } & 0\le x\le 2\,,\\[2ex]
              k+1 & \text{ if } & 2k-4\le x\le 2k-4+\dfrac{4}{k+2}\text{ for }3\le k<n\,,\\[2ex]
              \dfrac2k+\dfrac{k+2}{2k}x & \text{ if } & 2k-4+\dfrac{4}{k+2}\le x\le 2k-2\text{ for }3\le k \le n\,,\\[3ex]
              n+1 & \text{ if } & x\ge 2n-2\,.
              \end{array}
\right.
$$

\bigskip
\bigskip

\noindent\textit{Resultants.} Given $n\in\N$, $Q>1$ and $x\ge0$, consider the set
$$
\cR_{\le n}(Q,x):=\left\{(P_1,P_2)\in\Z[x]\times\Z[x]:\begin{array}{l}
2\le \deg P_i\le n, ~H(P_i)\le Q,\\[0.5ex]
1\le |R(P_1,P_2)|\le\rho Q^{x}
                                                     \end{array}
\right\}\,.
$$
By Theorem~\ref{t2}, there is a constant $\rho>0$ such that the number of pair of polynomials $P_1,P_2\in\cP_k(Q)$ lying in this set is
$$
\gg Q^{2k+2-2w}=Q^{x+2}\,,
$$
where $w$ is determined from the equation $x=2k-2w$ when $0\le 2w\le k+1$.
This gives the following restriction on $x$: $k-1\le x\le 2k$.
Consequently,
$$
\cR_{\le n}(Q,x)\gg Q^{x+2}\qquad\text{for}\qquad 0\le x\le 2n+2\,.
$$

\bigskip
\bigskip

\noindent\textit{Acknowledgements.}
The first and second authors are grateful to EPSRC for the support of their exchange visits through grant EP/J018260/1 and to SFB701 for supporting their visits to the University of Bielefeld.

{
\small

%---- plain!, alpha, unsrt, abbrv, siam!, amsalpha

  \def\polhk#1{\setbox0=\hbox{#1}{\ooalign{\hidewidth
  \lower1.5ex\hbox{`}\hidewidth\crcr\unhbox0}}} \def\cprime{$'$}
  \def\cprime{$'$}

}

\vspace*{1ex}

\begin{minipage}{\textwidth}
{\small Victor Beresnevich}\\
\footnotesize{\sc University of York, Heslington, York, YO10 5DD,
England}\\
{\it E-mail address}\,: \verb|victor.beresnevich@york.ac.uk|
\end{minipage}

\vspace*{1ex}

\begin{minipage}{\textwidth}
{\small Vasili Bernik}\\
\footnotesize{\sc Institute of mathematics, Surganova 11, Minsk, 220072, Belarus}\\
{\it E-mail address}\,: \verb|bernik@im.bas-net.by|, \ \ \verb|bernik.vasili@mail.ru|
\end{minipage}

\vspace*{1ex}

\begin{minipage}{\textwidth}
{\small Friedrich G\"otze}\\
\footnotesize{\sc University of Bielefeld, 33501, Bielefeld, Germany}\\
{\it E-mail address}\,: \verb|goetze@math.uni-bielefeld.de|
\end{minipage}

\end{document}